\documentclass[11pt, reqno]{article}

\usepackage{amsmath,amssymb,amsfonts,amsthm}
\usepackage{color}

\usepackage[colorlinks=true,linkcolor=blue,citecolor=blue,urlcolor=blue,pdfborder={0 0 0}]{hyperref}
\usepackage{cleveref}

\usepackage{CJKutf8}

\usepackage{caption}
\usepackage{thmtools}

\usepackage{mathrsfs}

\crefname{theorem}{Theorem}{Theorems}
\crefname{thm}{Theorem}{Theorems}
\crefname{lemma}{Lemma}{Lemmas}
\crefname{lem}{Lemma}{Lemmas}
\crefname{remark}{Remark}{Remarks}
\crefname{prop}{Proposition}{Propositions}
\crefname{defn}{Definition}{Definitions}
\crefname{corollary}{Corollary}{Corollaries}
\crefname{conjecture}{Conjecture}{Conjectures}
\crefname{question}{Question}{Questions}
\crefname{chapter}{Chapter}{Chapters}
\crefname{section}{Section}{Sections}
\crefname{figure}{Figure}{Figures}
\crefname{example}{Example}{Examples}

\theoremstyle{plain}
\newtheorem{thm}{Theorem}[section]

\newtheorem{lemma}[thm]{Lemma}
\newtheorem{theorem}[thm]{Theorem}

\newtheorem{corollary}[thm]{Corollary}

\newtheorem{prop}[thm]{Proposition}
\newtheorem{conjecture}[thm]{Conjecture}

\theoremstyle{definition}

\newtheorem{problem}[thm]{Problem}

\theoremstyle{remark}
\newtheorem{remark}[thm]{Remark}

\numberwithin{equation}{section}

\renewcommand{\P}{\mathbb P}
\newcommand{\E}{\mathbb E}

\newcommand{\Z}{\mathbb Z}
\newcommand{\N}{\mathbb N}

\usepackage[margin=0.9in]{geometry}
\usepackage{tikz}
\usepackage{pgfplots}
\usepackage{extarrows}
\usepackage{titling}
\usepackage{commath}
\usepackage{wasysym}
\usepackage{mathtools}
\usepackage{titlesec}
\newcommand{\eps}{\varepsilon}
\usepackage{bbm}
\usepackage{setspace}
\usepackage{enumitem}
\usepackage{booktabs}

\usepackage[textsize=tiny]{todonotes}

\usepackage{cite}

\newcommand{\Aut}{\operatorname{Aut}}

\def\P{\mathbb{P}}

\DeclareMathSymbol{\leqslant}{\mathalpha}{AMSa}{"36} 
\DeclareMathSymbol{\geqslant}{\mathalpha}{AMSa}{"3E} 
\DeclareMathSymbol{\eset}{\mathalpha}{AMSb}{"3F}     

\renewcommand{\epsilon}{\varepsilon}

\makeatletter
\pgfdeclareshape{slash underlined}
{
  \inheritsavedanchors[from=rectangle] 
  \inheritanchorborder[from=rectangle]
  \inheritanchor[from=rectangle]{north}
  \inheritanchor[from=rectangle]{north west}
  \inheritanchor[from=rectangle]{north east}
  \inheritanchor[from=rectangle]{center}
  \inheritanchor[from=rectangle]{west}
  \inheritanchor[from=rectangle]{east}
  \inheritanchor[from=rectangle]{mid}
  \inheritanchor[from=rectangle]{mid west}
  \inheritanchor[from=rectangle]{mid east}
  \inheritanchor[from=rectangle]{base}
  \inheritanchor[from=rectangle]{base west}
  \inheritanchor[from=rectangle]{base east}
  \inheritanchor[from=rectangle]{south}
  \inheritanchor[from=rectangle]{south west}
  \inheritanchor[from=rectangle]{south east}
  \inheritanchorborder[from=rectangle]
  \foregroundpath{
    \southwest \pgf@xa=\pgf@x \pgf@ya=\pgf@y
    \northeast \pgf@xb=\pgf@x \pgf@yb=\pgf@y
    \pgf@xc=\pgf@xa
    \advance\pgf@xc by .5\pgf@xb
    \pgf@yc=\pgf@ya
    \advance\pgf@xc by -1.3pt
    \advance\pgf@yc by -1.8pt
    \pgfpathmoveto{\pgfqpoint{\pgf@xc}{\pgf@yc}}
    \advance\pgf@xc by  2.6pt
    \advance\pgf@yc by  3.6pt
    \pgfpathlineto{\pgfqpoint{\pgf@xc}{\pgf@yc}}
    \pgfpathmoveto{\pgfqpoint{\pgf@xa}{\pgf@ya}}
    \pgfpathlineto{\pgfqpoint{\pgf@xb}{\pgf@ya}}
 }
}
\tikzset{nomorepostaction/.code=\let\tikz@postactions\pgfutil@empty}

\makeatother

\makeatletter
\DeclareFontFamily{OMX}{MnSymbolE}{}
\DeclareSymbolFont{MnLargeSymbols}{OMX}{MnSymbolE}{m}{n}
\SetSymbolFont{MnLargeSymbols}{bold}{OMX}{MnSymbolE}{b}{n}
\DeclareFontShape{OMX}{MnSymbolE}{m}{n}{
    <-6>  MnSymbolE5
   <6-7>  MnSymbolE6
   <7-8>  MnSymbolE7
   <8-9>  MnSymbolE8
   <9-10> MnSymbolE9
  <10-12> MnSymbolE10
  <12->   MnSymbolE12
}{}
\DeclareFontShape{OMX}{MnSymbolE}{b}{n}{
    <-6>  MnSymbolE-Bold5
   <6-7>  MnSymbolE-Bold6
   <7-8>  MnSymbolE-Bold7
   <8-9>  MnSymbolE-Bold8
   <9-10> MnSymbolE-Bold9
  <10-12> MnSymbolE-Bold10
  <12->   MnSymbolE-Bold12
}{}

\let\llangle\@undefined
\let\rrangle\@undefined
\DeclareMathDelimiter{\llangle}{\mathopen}%
                     {MnLargeSymbols}{'164}{MnLargeSymbols}{'164}
\DeclareMathDelimiter{\rrangle}{\mathclose}%
                     {MnLargeSymbols}{'171}{MnLargeSymbols}{'171}
\makeatother

\pretitle{\begin{flushleft}\Large}
\posttitle{\par\end{flushleft}\vskip 0.5em
}
\preauthor{\begin{flushleft}}
\postauthor{
\\
\vspace{0.5em}
\footnotesize{The Division of Physics, Mathematics and Astronomy, California Institute of Technology\\ Email: 
\href{mailto:t.hutchcroft@caltech.edu}{t.hutchcroft@caltech.edu}}
\end{flushleft}}
\predate{\begin{flushleft}}
\postdate{\par\end{flushleft}}
\title{\bf Small-ball estimates for random walks on groups} 

\renewenvironment{abstract}
 {\par\noindent\textbf{\abstractname.}\ \ignorespaces}
 {\par\medskip}

\titleformat{\section}
  {\normalfont\large \bf}{\thesection}{0.4em}{}

\author{{\bf Tom Hutchcroft}}

\begin{document}

\date{\small{\today}}

\maketitle

\setstretch{1.1}

\begin{abstract}
We prove a new inequality bounding the probability that the random walk on a group has small total displacement in terms of the spectral and isoperimetric profiles of the group. This inequality implies that if the random walk on the group is diffusive then Cheeger's inequality is sharp in the sense that the isoperimetric profile $\Phi$ and spectral profile $\Lambda$ of the group are related by $\Lambda \simeq \Phi^2$. Our inequality also yields substantial progress on a conjecture of Lyons, Peres, Sun, and Zheng (2017) stating that for any transient random walk on an infinite, finitely generated group, the expected occupation time of the ball of radius $r$ is $O(r^2)$: We prove that this conjecture holds for every group of superpolynomial growth whose
   spectral profile is slowly varying, which we conjecture is always the case. For groups of exponential or stretched-exponential growth satisfying a further mild regularity assumption on their spectral profile, our method yields the strong quantitative small-ball estimate
\[-\log \P\bigl(d(X_0,X_n) \leq \eps n^{1/2}\bigr) \succeq \frac{1}{\eps^2} \wedge (-\log \P(X_n=X_0)),\] which is sharp for the lamplighter group. Finally, we prove that the regularity assumptions needed to apply the strongest versions our results are satisfied for several classical examples where the spectral profile is not known explicitly, including the first Grigorchuk group and Thompson's group $F$.
\end{abstract}

\section{Small-ball estimates and their consequences}
\label{sec:intro}

In this paper we prove the following inequality relating three of the most fundamental quantities describing the large-scale geometry of a group: the \emph{isoperimetric profile}, the \emph{spectral profile}, and the \emph{rate of escape} of the random walk. 

\begin{theorem}
\label{thm:main_inequality}  Let $G=(V,E)$ be a connected, locally finite, unimodular transitive graph, let $X=(X_k)_{k\geq 0}$ be the simple random walk on $G$, and let $\Phi$ and $\Lambda$ denote the isoperimetric profile and spectral profile of $G$ respectively. The inequality
\begin{align*}
\P\left(d(X_0,X_k) \leq r\right) 
&\leq 2 \exp\left[ - \frac{\log 2}{2}\max\left\{\ell \geq 0 :  \frac{\ell \log 2}{\Lambda(2^{\ell+1}\Phi^{-1}(c/r))} \leq k \right\} \right]
\end{align*}
holds for 
 every pair of integers $k,r \geq 1$, where $c=\min\{\Aut(G) e \cap E^\rightarrow_o : e\in E^\rightarrow\}/2\deg(o)$.
\end{theorem}

Here we think of each edge of $G$ as consisting of a pair of oriented edges in opposite directions, write $E^\rightarrow$ for the set of oriented edges of $G$, write $E^\rightarrow_o$ for the set of oriented edges emanating from some base point $o\in V$, and write $\Gamma e$ for the orbit of $e$ under the automorphism group of $G$. In particular, the constant $c$ appearing in this theorem is uniformly  bounded away from zero for edge-transitive graphs and for graphs of bounded degree. \emph{Unimodularity} is a technical condition which holds for all amenable transitive graphs and all Cayley graphs of finitely generated groups \cite{MR1082868}.
An extension to general (not necessarily nearest-neighbour) random walks is given in \cref{thm:main_inequality_general}.

In the remainder of this section we will state all the remaining definitions needed to understand this theorem and explain its consequences towards various open problems in the area, referring the reader to e.g.\ \cite{zheng2022asymptotic,LP:book,Woess,yadin2024harmonic,Pete} for more detailed background and to the introduction of \cite{saloff2021isoperimetric} for a historical account.



\subsection{Background and definitions}

To simplify notation, we give definitions under the assumption that the transition matrix is symmetric and hence that the counting measure is stationary, which holds automatically in the transitive case.

\medskip

\noindent \textbf{Spectral and isoperimetric profiles.}  Let $P:V\times V\to [0,1]$ be a stochastic matrix defined on a countable set $V$ which is symmetric in the sense that $P(x,y)=P(y,x)$ for every $x,y\in V$. The \textbf{isoperimetric profile} (a.k.a.\ $L^1$ isoperimetric profile) of $P$ is defined by
\[
\Phi(n)=\Phi_P(n) := \inf \left\{ \frac{1}{|\Omega|} \sum_{x\in \Omega, y \in V\setminus \Omega} P(x,y) : \Omega \subset V, |\Omega|\leq n \right\}.
\]
The \textbf{spectral profile} (a.k.a.\ $L^2$ isoperimetric profile) of $P$ is defined by
\begin{equation*}
\Lambda(n) = \Lambda_P(n) : = \inf\left\{\lambda_P(\Omega) : \Omega \subset V, |\Omega| \leq n\right\}
\end{equation*}
where
\begin{equation*}
\lambda_P(\Omega) := \inf\left\{\langle(I-P)f,f\rangle : \text{support}(f) \subset \Omega, \|f\|_2 = 1\right\}
\end{equation*}
is the spectral gap of the chain killed upon leaving $\Omega$. One can also give an analogous definition of the isoperimetric profile in terms of $L^1$ operator norms for the killed random walk, see e.g.\ \cite[Section 3.1]{zheng2022asymptotic}. When $G$ is a regular graph, we refer to the isoperimetric and spectral profiles of the simple random walk transition matrix on $G$ as the isoperimetric and spectral profiles of $G$, and omit the subscript $P$ when this does not cause confusion. 

\medskip

\noindent \textbf{Equivalent rates of growth and decay.}
Although the spectral and isoperimetric profiles are defined for \emph{graphs}, they can also be thought of as quantities associated to a \emph{group} provided that one considers them modulo appropriate notions of asymptotic equivalence. Given two monotone, non-negative real functions $f$ and $g$ defined for all sufficiently large positive real inputs, we write $f\simeq g$ if there exist positive constants $c_1,c_2,c_3,c_4>0$ and $x_0<\infty$ such that
\[
c_1f( c_2 x) \leq g(x) \leq c_3 f(c_4 x)
\]
for all $x\geq x_0$. We also write $\lesssim$ and $\gtrsim$ for the one-sided inequality versions of this relation. This notion of equivalence distinguishes between e.g.\ different powers of $x$, different powers of $\log x$, and exponentials of different powers of $x$, but does not distinguish between e.g.\ exponentials with different constants such as $2^x$ and $3^x$. It is easily proven that if $\Phi_1,\Phi_2$ and $\Lambda_1,\Lambda_2$ are the isoperimetric and spectral profiles associated to two different symmetric generating sets of some finitely generated group $\Gamma$ then $\Phi_1\simeq \Phi_2$ and $\Lambda_1 \simeq \Lambda_2$. Given a decreasing function $f$ we will write $f^{-1}(x):=\inf\{t:f(t)\leq x\}$, while for increasing functions we will write $f^{-1}(x)=\inf\{t:f(t)\geq x\}$.

\medskip

\noindent \textbf{Relation to volume growth and heat kernel decay.} 
The spectral and isoperimetric profiles are also closely related to various other important quantitative features of groups, including the \emph{volume growth} and \emph{heat kernel decay}. Indeed, it is a classical theorem of Coulhon and Saloff-Coste \cite{MR1232845} (extended to unimodular transitive graphs in \cite{LMS08} and non-unimodular transitive graphs in \cite{tessera2020sharp}, with different constants) that the isoperimetric profile can be lower bounded in terms of the \emph{inverse growth function} by
\begin{equation}
\label{eq:Coulhon_Saloff_Coste}
\Phi(n) \geq \frac{1}{2\deg(o) \operatorname{Gr}^{-1}(2n)},
\end{equation}
where $\operatorname{Gr}(r)$ denotes the number of points in the ball of radius $r$ and $\operatorname{Gr}^{-1}(n)=\inf\{r\geq 0:\operatorname{Gr}(r)\geq n\}$. (See also \cite{pittet2022coulhon,correia2024isoperimetry} for optimized forms of this inequality.)
On the other hand, we have trivially that $\operatorname{Gr}(n+1)\geq (1+\Phi(\operatorname{Gr}(n)) \operatorname{Gr}(n)$ for every $n\geq 0$,
which implies by a small calculation that there exists a universal constant $C$ such that
\begin{equation}
\label{eq:Growth_Bounds_Isoperimetry}
\Phi(n) \leq \frac{C}{ \operatorname{Gr}^{-1}(n)-\operatorname{Gr}^{-1}(n/2)}.
\end{equation}
 For transitive graphs of \emph{polynomial growth}, $\operatorname{Gr}^{-1}(n)-\operatorname{Gr}^{-1}(n/2)$ is of the same order as $\operatorname{Gr}^{-1}(2n)$ for large $n$ and one can use these inequalities to argue that $\Phi(n) \asymp n^{-1/d}$, where $d$ is the integer volume growth dimension whose existence is guaranteed by the structure theory of transitive graphs of polynomial growth \cite{MR623534,MR811571}. For graphs of superpolynomial growth, the connection between the growth and isoperimetric profile is less tight, and we have e.g.\ that
\begin{align*}
\Biggl(\operatorname{Gr}(n) \gtrsim \exp\left[n^\alpha\right] \Biggr) &\Rightarrow \Biggl(\Phi(n)\gtrsim (\log n)^{-1/\alpha}  \Biggr) \Rightarrow \Biggl(\operatorname{Gr}(n) \gtrsim \exp\left[ n^{\alpha/(1+\alpha)}\right]\Biggr)
\end{align*}
for each $\alpha > 0$. 
Among groups of exponential growth, nonamenable groups such as free groups have non-decaying isoperimetric profile while the lamplighter group over $\Z$ has $\Phi(n)\simeq 1/\log n$, so that both constraints above are sharp.
Erschler \cite{MR2254627} has shown that groups of subexponential growth can have isoperimetric profiles of arbitrarily slow decay, while Brieussel and Zheng \cite{brieussel2015speed} have shown that groups of exponential growth can have essentially arbitrary decay of their isoperimetric profile between non-decay and $1/\log n$ decay. 
(The relation $\Lambda \simeq \Phi^2$ holds for all the examples constructed in \cite{brieussel2015speed}.)

It has also been shown in a series of works by Coulhon and Grigor’yan \cite{MR1418518,coulhon1997diagonal} that the asymptotic decay of return probabilities (a.k.a.\ the heat kernel) $P^{2n}(o,o)$  and the asymptotic decay of the spectral profile $\Lambda$  determine one another under mild regularity assumptions\footnote{More generally, if one quantity admits one-sided bounds in terms of a ``sufficiently regular'' function then the other quantity admits analogous one-sided bounds via the integral transform of that function defined implicitly in \eqref{eq:Grigoryan}.} via
\begin{equation}
\label{eq:Grigoryan}
P^{2n}(o,o) \simeq \frac{1}{\psi(2n)} \qquad \text{ where  $\psi$ is defined  implicitly by } \qquad t=\int_1^{\psi(t)} \frac{1}{x\Lambda(x)} \dif x.
\end{equation}
If one makes the stronger regularity assumption that $\Lambda(2^n)$ is \emph{doubling} (i.e.\ satisfies $\Lambda(2^{2n}) \geq c\Lambda(2^n)$ for some $c>0$ and every $n\geq 1$), as is the case when e.g.\ $\Lambda(x) \simeq (\log x)^{-\alpha}$, then this relationship may be expressed more succinctly as 
\begin{equation}
\label{eq:HeatKernel}
P^{2n}(o,o) \simeq \exp[-\Psi^{-1}(n)] \qquad \text{ where } \qquad \Psi(x):= \frac{x}{\Lambda(2^x)}.
\end{equation}
See also \cite{bendikov2012spectral} for a more precise asymptotic analysis of the heat kernel under this doubling condition.
For example, the $d$-dimensional lamplighter group $\Z_2 \wr \Z^d$ has $\Lambda(n)\simeq (\log n)^{-2/d}$ and $P^{2n}(o,o) \simeq \exp[-n^{d/(d+2)}]$. The relation \eqref{eq:HeatKernel} can be thought of as a refinement of Kesten's theorem \cite{Kesten1959b}, which states that $G$ is amenable if and only if $P^{2n}(o,o)$ decays subexponentially. Let us also mention that the decay rate of $P^{2n}(o,o)$ was proven to be a quasi-isometry invariant of groups by Pittet and Saloff-Coste \cite{pittet2000stability}.

\medskip
\noindent\textbf{Rates of escape.}
A further important quantity associated to a graph is the \emph{rate of escape} of the random walk, i.e., the asymptotic rate of growth of $d(X_0,X_n)$ where $X$ is the random walk and $d$ denotes the graph distance. In contrast to the other quantities we have discussed, it is a major open problem whether this is a \emph{group quantity} in the sense that different Cayley graphs of the same group have the same rate of escape for random walk. The stability problem is open even for  \emph{ballistic} (a.k.a.\ \emph{positive speed}) case $d(X_0,X_n)\asymp n$, where the rate of escape estimate is equivalent to non-triviality of the Poisson boundary \cite{avez1976harmonic,derriennic1980quelques,kaimanovich1983random}. Similarly, one can formulate various different precise quantities describing the rate of escape (e.g. the mean, median, and $L^2$ norm of $d(X_0,X_n)$), and it is open in general to prove that these different notions are $\simeq$-equivalent. See e.g.\ \cite[Section 3.3]{zheng2022asymptotic} and \cite[Chapter 13]{LP:book} for further background.

\medskip

\noindent \textbf{Relation to other quantities.}
The growth, escape rate, and heat kernel decay are related via the elementary inequality
\[
P^{2n}(o,o)=\sum_{v\in V}P^n(o,v)^2 \geq \sum_{d(o,v)\leq r}P^n(o,v)^2 \geq \frac{\P(d(X_0,X_n)\leq r)^2}{\operatorname{Gr}(r)},
\]
where the final inequality follows by Cauchy-Schwarz. In particular, if $\mathcal{E}(n)$ denotes the median of $d(X_0,X_n)$ then $P^{2n}(o,o) \gtrsim 1/\operatorname{Gr}(\mathcal{E}(n))$. It is much less obvious to what extent one can \emph{upper bound} return probabilities in terms of the growth and rate of escape. These quantities are more closely related to the \emph{entropy} $\E [-\log P^n(X_0,X_n)]$, which in turn satisfies some (fairly weak) two-sided inequalities relating it to $P^{2n}(o,o)$ as established in \cite{doi:10.1093/imrn/rny034}.
 The rate of escape of the random walk on a group is also closely related to the \emph{metric-embeddability} properties of the group \cite{austin2009wreath}: for example, any group admitting a bi-Lipschitz equivariant embedding into Hilbert space must have a diffusive random walk. See e.g.\ \cite[Section 3.4]{zheng2022asymptotic} for further background. 

\medskip
\noindent \textbf{Diffusive lower bounds.} It is a classical fact that the simple random walk on $\Z^d$ is \emph{diffusive}, meaning that $|X_n|$ is typically of order $n^{1/2}$ when $n$ is large, and in fact the same is true for any infinite, finitely generated group of polynomial growth by a theorem of Hebisch and Saloff-Coste \cite{HebSaCo93}. It was proven by Lee and Peres \cite{MR3127886}, building on an unpublished argument of Erschler, that this is in fact the \emph{slowest possible} rate of escape for the random walk on an infinite, finitely generated group: If $X$ is the random walk on any infinite, connected, locally finite transitive graph then $\E d(X_0,X_n) \gtrsim n^{1/2}$. (Note that this does \emph{not} follow from the methods discussed in the previous paragraph, which only yield an $n^{1/3}$ lower bound for groups of exponential volume growth.) This is in stark contrast to random walks on \emph{fractals}, which are often subdiffusive \cite{KumagaiBook}. The Lee-Peres proof works by case analysis according to whether or not the graph is amenable: In the nonamenable case the walk trivially has positive speed by Kesten's theorem, while in the amenable case the graph always admits an equivariant harmonic embedding into Hilbert space and the claim can be proven via martingale analysis. (Indeed, the weaker statement that $\E [d(X_0,X_n)^2]\gtrsim n$ follows immediately by the orthogonality of martingale increments.) Note that the diffusive lower bound can still be sharp for groups of exponential growth, as it is for the lamplighter group $\Z_s \wr \Z$, and Brieussel and Zheng \cite{brieussel2015speed} have shown that groups of exponential growth can have essentially arbitrary rates of escape between $n^{1/2}$ and $n$.

\subsection{Consequences of \cref{thm:main_inequality}}

\medskip
\noindent \textbf{When is Cheeger's inequality saturated?} 
A fundamental relationship between these two notions of isoperimetry, known as \textbf{Cheeger's inequality}, states that
\begin{equation}
\label{eq:Cheeger}
\frac{1}{2}\Phi^2 \leq \Lambda \leq \Phi.
\end{equation}
A graph is said to be \textbf{amenable} if $\lim_{n\to\infty} \Phi(n)=\lim_{n\to\infty}\Lambda(n)=0$, where these two conditions are equivalent by Cheeger's inequality. In all known examples of infinite transitive graphs it is in fact the case that the \emph{lower bound} of \eqref{eq:Cheeger} is sharp in the sense that $\Lambda$ is of order $\Phi^2$. This suggests the following two closely related problems:


\begin{problem}
\label{prob:PittetSaloffCoste}
Does the asymptotic relation $\Lambda \simeq \Phi^2$ hold for simple random walk on every finitely generated group?
\end{problem}

\begin{problem}
\label{prob:PittetSaloffCoste2}
Do $\Phi$ and $\Lambda$ define the same quasi-isometry invariant of groups, in the sense that two groups have $\Phi_{G_1}\simeq \Phi_{G_2}$ if and only if $\Lambda_{G_1}\simeq \Lambda_{G_2}$?
\end{problem}

\Cref{prob:PittetSaloffCoste} was posed by Pittet in \cite{pittet2000isoperimetric}, while 
\Cref{prob:PittetSaloffCoste2} is often credited to Pittet and Saloff-Coste, who posed the problem as a conjecture in \cite{pittet1999amenable}. According to Saloff-Coste [personal communication], both problems have been around as folklore since the mid 1980's.
As far as we can tell, there is no widespread consensus in the community as to whether \Cref{prob:PittetSaloffCoste}  should admit a positive solution or not; Tianyi Zheng [personal communication] has suggested it may be false for certain self-similar groups. \Cref{prob:PittetSaloffCoste2} appears to be regarded as more likely to be true, although the results of Brieussel and Zheng \cite{brieussel2015speed} show that the two problems are essentially equivalent in the case that $\Phi$ has at most $1/\log n$ decay. (Indeed, they construct groups of exponential growth which have essentially arbitrary decay of $\Phi$ between $1/\log n$ and non-decay, all of which satisfy the relation $\Lambda \simeq \Phi^2$.)


\medskip

 \cref{thm:main_inequality} allows us to resolve Problems \ref{prob:PittetSaloffCoste} and \ref{prob:PittetSaloffCoste2} positively in the case that the random walk is diffusive. 

\begin{corollary}
\label{cor:PittetSaloffCoste}
Let $G=(V,E)$ be a connected, locally finite, transitive graph and let $\Phi$ and $\Lambda$ denote the isoperimetric profile and spectral profile of $G$ respectively. If the random walk on $G$ is diffusive in the sense that $\liminf_{n\to\infty}\P(d(X_0,X_n)\leq C n^{1/2})>0$ for some $C<\infty$ then $\Lambda \simeq \Phi^2$.
\end{corollary}

\begin{proof}[Proof of \cref{cor:PittetSaloffCoste} given \cref{thm:main_inequality}]
If $G$ is nonamenable then the claim is vacuous, so we may assume that it is amenable and hence unimodular.
The inequality $\Lambda \gtrsim \Phi^2$ follows from Cheeger's inequality \eqref{eq:Cheeger}, so it suffices to prove that $\Lambda \lesssim \Phi^2$ under the diffusive assumption.
Suppose then that the diffusive assumption $\liminf_{n\to\infty}\P(d(X_0,X_n)\leq C n^{1/2})>0$ holds for some $C<\infty$ and let $c_1$ be the constant from \cref{thm:main_inequality}.
Applying \cref{thm:main_inequality} with $r=Cn^{1/2}$ and letting $c_2=c_1/C$, we obtain that
\[
\sup_n \max\left\{\ell\geq 0: n\geq \frac{\ell \log 2}{\Lambda(2^{\ell+1}\Phi^{-1}(c_2n^{-1/2}))}\right\} < \infty,
\]
or in other words that there exists $\ell_0<\infty$ such that
\[
\frac{\ell_0 \log 2}{\Lambda(2^{\ell_0+1}\Phi^{-1}(c_2n^{-1/2}))} > n
\]
for every $n\geq 1$. Rearranging this inequality yields that there exist positive constants $C_2$ and $C_3$ such that
\begin{equation}
\label{eq:Cheeger_rearrangement}
\Lambda(C_2 \Phi^{-1}(\eps)) \leq C_3 \eps^2 \qquad \text{ and hence that } \qquad C_2 \Phi^{-1}(\eps) \leq \Lambda^{-1}(C_3 \eps^2)
\end{equation}
for every $0<\eps\leq 1$ (values of $\eps$ not of the form $c_2n^{-1/2}$ can be dealt with by picking the minimal $n$ such that $c_2n^{-1/2}\leq \eps$ and using that $\Lambda$ and $\Phi^{-1}$ are decreasing and increasing respectively). We claim that this implies that the inequality
\[
\Lambda (C_2 n) \leq C_3 \Phi(n)^2
\]
holds for all $n\geq 1$. Suppose that this is not the case, let $n$ be such that $\Lambda (C_2 n) > C_3 \Phi(n)^2$, and let $\eps=\Phi(n)<1$. We have by definition that $\Phi^{-1}(\eps)\leq n$ and $\Lambda^{-1}(C_3 \eps^2)> C_2n$, contradicting \eqref{eq:Cheeger_rearrangement}. This completes the proof.
\end{proof}

\begin{remark}
 If \Cref{prob:PittetSaloffCoste} has a negative answer, the maximal exponent $\alpha$ for which $\Lambda \lesssim \Phi^\alpha$ would become an interesting feature of a group to study. The proof of \cref{cor:PittetSaloffCoste} shows that this exponent satisfies $\alpha \geq 1/\beta$ where $\beta$ is infimal such that $\liminf_{n\to\infty}\P(d(X_0,X_n)\leq C n^{\beta})>0$ for some $C<\infty$.  In particular, if there exist amenable groups in which $\Lambda \simeq \Phi$ then the random walks on these groups must have positive speed. More generally, the proof of \cref{cor:PittetSaloffCoste} shows more generally that if $d(X_0,X_n)$ is of order at most $f(n)$ with good probability for some increasing function $f$ then $\Lambda \lesssim f^{-1}(\Phi)$.
\end{remark}

\noindent \textbf{Occupation measures.} Our next application of \cref{thm:main_inequality} concerns a natural strengthening of the Lee-Peres theorem \cite{MR3127886} on diffusive lower bounds. While the methods of \cite{MR3127886} establish that $\E d(X_0,X_n) = \Omega(n^{1/2})$, they give only rather weak bounds on the probability that $d(X_0,X_n)$ is much smaller than $n^{1/2}$: The original Lee-Peres paper established the Ces\`aro estimate 
\[\frac{1}{n}\sum_{i=1}^n \P\bigl(d(X_0,X_i) \leq r\bigr) =O\left(\frac{r}{\sqrt{n}}\right),\]
which was improved to the pointwise estimate
\begin{equation}
\label{eq:martingale_small_ball} \P\bigl(d(X_0,X_n) \leq r\bigr) =O\left(\frac{r}{\sqrt{n}}\right)
\end{equation}
using martingale small-ball estimates in \cite{lee2016gaussian}. While these estimates are sharp for the simple random walk on $\Z$, they are presumably very far from the truth for groups of superpolynomial growth: For groups of $d$-dimensional polynomial growth the true probability is of order $(r/\sqrt{n})^d$ \cite{HebSaCo93}, so that a $(r/\sqrt{n})^{\omega(1)}$ bound should presumably be possible for groups of superpolynomial growth. 

\medskip

Our lack of understanding of this matter was highlighted in the work of Lyons, Peres, Sun, and Zheng \cite{lyons2020occupation} (see also Zheng's ICM proceedings \cite{zheng2022asymptotic}), who made the following conjecture on the expected \emph{occupation measure} of a ball:

\begin{conjecture}
\label{conj:occupation} Let $\Gamma$ be a finitely generated group and let $\mu$ be a symmetric measure on $\Gamma$ whose support generates $\Gamma$. Let $Z_1,Z_2,\ldots$ are independent, identically distributed random variables each with law $\mu$ and let $X$ denote the random walk $X_n = Z_1 Z_2\cdots Z_n$. If $X$ is transient then
\[
\E \left[\sum_{k=0}^\infty \mathbbm{1}(d(X_0,X_n)\leq r) \right] = O(r^2) \qquad \text{as $r\to\infty$.}
\]
\end{conjecture}

\begin{remark}
The conjecture stated in \cite{lyons2020occupation,zheng2022asymptotic} does not include the hypothesis that the support of $\mu$ generates $\Gamma$. However, the conjecture is false without this assumption. Let $\Gamma$ be the lamplighter group $\Z_2 \wr \Z$, and consider the random walk which, at each step, picks a random non-negative integer $N$ with some fixed distribution $\nu$ then independently randomizes the states of the lamps on the interval $[-N,N]$; the $\Z$ coordinate of the lamplighter remains fixed throughout. If $N_1,N_2,\ldots$ denote the i.i.d.\ random integers used to generate the walk, then the distance to the origin in the lamplighter group at time $n$ is $O(\max\{N_i : 1\leq i\leq n\})$, and can therefore be made to grow in an essentially arbitrary way by picking the measure $\nu$ appropriately. On the other hand, conditional on the random sequence $N$, the probability that the walk is at the origin at time $n$ is $2^{-2\max\{N_i : 1\leq i\leq n\}-1}$, so that the walk is transient whenever $\max\{N_i : 1\leq i\leq n\}$ grows at least like $C \log n$ for an appropriately large constant $C$. Thus, walks whose support does not generate a finitely generated group can be much slower than diffusive while remaining transient.
\end{remark}

This conjecture is known to hold whenever $\Gamma$ has polynomial growth. Indeed, it is a consequence of Gromov's theorem \cite{MR623534} and the Bass-Guivarc'h formula \cite{bass72poly-growth,guivarch73poly-growth} that every group of polynomial growth has an \emph{integer} volume growth dimension $d$ such that $\operatorname{Gr}(r) \simeq r^d$. If $\Gamma$ has dimension at least three, \cref{conj:occupation} follows from the fact that every random walk on $\Gamma$ whose support generates $\Gamma$ satisfies the heat kernel bound $P^{n}(x,y)\lesssim n^{-d/2}$. (This bound follows from \eqref{eq:Coulhon_Saloff_Coste} together with the relationships between the isoperimetric profile and heat kernel decay.) Meanwhile, if $\Gamma$ has dimension $1$ or $2$ and the walk $X$ is transient then
\[
\E \left[\sum_{k=0}^\infty \mathbbm{1}(d(X_0,X_n)\leq r) \right] \leq \operatorname{Gr}(r) \sum_{k=0}^\infty \sup_{x,y} P^k(x,y) = O(\operatorname{Gr}(r)) = O(r^2),
\] 
where the sum in the middle expression converges by the assumption that the walk is transient and symmetric. Thus, \cref{conj:occupation} is open only for groups of superpolynomial growth.

In unpublished work of Peres and Zheng discussed in \cite[Section 4.2.2]{zheng2022asymptotic}, it is proven that if a group does not have \emph{Shalom's property} $H_\mathrm{FD}$ \cite{shalom2004harmonic} then the martingales arising from its equivariant harmonic embedding into Hilbert space have an ``asymptotic orthogonality'' property which allows one to improve the bound \eqref{eq:martingale_small_ball} to
\begin{equation}
\label{eq:martingale_small_ball} \P\bigl(d(X_0,X_n) \leq r\bigr) \leq C_p \left(\frac{r}{\sqrt{n}}\right)^p
\end{equation}
for any exponent $p<\infty$ via appropriate martingale small-ball estimates; this estimate is more than enough to prove \cref{conj:occupation}. (We will not define this property here since we do not use it.)
On the other hand, the best known result for an \emph{arbitrary} group of superpolynomial growth, proven in \cite{lyons2020occupation}, is that the simple random walk satisfies
\[\E \left[\sum_{k=0}^\infty \mathbbm{1}(d(X_0,X_k)\leq r) \right] = O(r^2\sqrt{\log \operatorname{Gr}(r)})=O(r^{5/2}).\]
It is also proven in the same paper that if $X$ is the simple random walk on a finitely generated group of superpolynomial growth then
\begin{equation}
\label{eq:WUSF}\E \left[\sum_{k=0}^\infty k \mathbbm{1}(d(X_0,X_k)\leq r) \right] = O(r^4(\log \operatorname{Gr}(r))^{3/2})=O(r^{11/2}),\end{equation}
which has interpretations in terms of the occupation measure of \emph{wired uniform spanning forest} components on the Cayley graph; they conjectured that in fact this quantity is always $O(r^4)$.

\medskip

As we will see, \cref{thm:main_inequality} gets us frustratingly close to a full resolution of \cref{conj:occupation}. Before stating exactly what we are able to prove, let us first state the following generalization of \cref{conj:occupation} that allows for random walks other than the simple random walk. 

\begin{theorem}
\label{thm:main_inequality_general}
Let $G=(V,E)$ be a connected, locally finite, transitive graph, let $\Gamma \subseteq \Aut(G)$ be a closed transitive, unimodular subgroup of $\Aut(G)$, and suppose that $Q:V\times V\to[0,1]$ is a stochastic matrix that is symmetric and $\Gamma$-diagonally invariant in the sense that $Q(\gamma x, \gamma y)=Q(x,y)$ for every $x,y\in V$ and $\gamma \in \Gamma$. Let $X=(X_k)_{k\geq 0}$ be the random walk on $V$ with transition matrix $Q$. 
The inequality
\begin{align*}
\P\left(d(X_0,X_k) \leq r\right) 
&\leq 2\exp\left[ - \frac{\log 2}{2}\max\left\{\ell \geq 0 : \frac{\ell \log 2}{\Lambda_Q(2^{\ell+1}\Phi^{-1}(c/r))} \leq k \right\} \right]
\end{align*}
holds for 
 every pair of integers $k,r \geq 1$, where $c=\min\{\Gamma e \cap E^\rightarrow_o : e\in E^\rightarrow\}/2\deg(o)$.
\end{theorem}

\begin{remark}
In this theorem, $\Lambda_Q$ denotes the spectral profile with respect to $Q$ but $\Phi$ still denotes the isoperimetric profile associated to the simple random walk on $G$. If $Q(x,y)>0$ for every pair of neighbouring vertices in $G$ then we can write $Q=\theta P +(1-\theta)Q'$ for some stochastic matrix $Q'$ and $\theta\in (0,1)$, where $P$ denotes the simple random walk transition matrix on $G$, and it follows that $\Lambda_Q \geq \theta \Lambda$. When $Q$ describes a walk that takes large jumps on $G$, the spectral profile $\Lambda_Q$ may be much larger than the spectral profile $\Lambda$ of $G$ \cite{saloff2016random,saloff2015random}.
\end{remark}

Let us now explain the relevance of \cref{thm:main_inequality,thm:main_inequality_general} to \cref{conj:occupation}.
Morally, \cref{thm:main_inequality_general} should lead to diffusive lower bounds since, by Cheeger's inequality, 
\begin{equation}
\label{eq:dubious}
\Lambda_Q(C \Phi^{-1}(c/r)) \geq c \Lambda(C \Phi^{-1}(c/r)) \geq \frac{c}{2} \Phi(C \Phi^{-1}(c/r))^2 \text{``$\geq$''} c r^{-2},\end{equation}
where $c$ and $C$ denote constants that can change from one expression to the next. Moreover, we know that if the graph has superpolynomial growth then $\Phi$ decays subpolynomially by the Coulhon--Saloff-Coste bound \eqref{eq:Coulhon_Saloff_Coste}, so that $\Lambda(2^\ell \Phi^{-1}(c_1/r))$ ``should'' be of the same order as $\Lambda(\Phi^{-1}(c_1/r))$ when $r$ is large and $\ell$ is not too large. Unfortunately, we do not currently know how to fully justify the previous sentence. In fact, we have not even been able to prove the doubling property $\Phi(2n)\geq C \Phi(n)$ needed to justify the final inequality of \eqref{eq:dubious} without making some assumptions on the group. 

\medskip

The following corollary verifies \cref{conj:occupation} under a mild regularity assumption on the spectral and isoperimetric profiles.
  In fact we prove a stronger statement bounding higher moments of occupation measures, the case $p=2$ of which also establishes the sharp form of \eqref{eq:WUSF} subject to the same caveats.  Here, we recall that a positive real-valued function is said to be \textbf{slowly varying} if $\lim_{t\to\infty} \sup_{\lambda \in [1,2]} f(\lambda t)/f(t)=\lim_{t\to\infty} \inf_{\lambda \in [1,2]} f(\lambda t)/f(t)=1$ and say that a function is \textbf{roughly slowly varying} if it is $\simeq$-equivalent to a slowly varying function.

\begin{corollary}
\label{cor:occupation_measure}
Let $G=(V,E)$ be a connected, locally finite, transitive graph, let $\Gamma \subseteq \Aut(G)$ be a closed, transitive, unimodular subgroup of $\Aut(G)$, and suppose that $Q:V\times V\to[0,1]$ is a stochastic matrix that is symmetric and $\Gamma$-diagonally invariant in the sense that $Q(\gamma x, \gamma y)=Q(x,y)$ for every $x,y\in V$ and $\gamma \in \Gamma$. 
   If there exists $\beta>0$ and a roughly slowly varying function $\tilde \Lambda$ such that $\Phi^\beta \lesssim \tilde \Lambda \lesssim \Lambda_Q$, then for each $p\in [1,\infty)$ there exists a constant $C_p<\infty$ such that
  \[
\P\bigl(d(X_0,X_k) \leq r\bigr) \leq C_p \left(\frac{r}{k^{1/\beta}}\right)^p
  \]
  for every $k,r\geq 1$. As a consequence,
  for each $p\in [1,\infty)$ there exists a constant $C_p'<\infty$  such that \[
   \sum_{k=0}^\infty (k+1)^{p-1} \P(d(X_0,X_k)\leq r) \leq C_p' r^{\beta p}\]
  for every $r\geq 1$.
\end{corollary}

Note that the quantity $\sum_{k=0}^\infty (k+1)^{p-1} \P(d(X_0,X_k)\leq r)$ is an upper bound on the $p$th moment of the occupation measure of the ball of radius $r$ up to a constant of order $p$.
By Cheeger's inequality \eqref{eq:Cheeger}, if 
 $Q(x,y)>0$ whenever $x$ and $y$ are adjacent in $G$ (as is certainly the case for the simple random walk) and 
any of the functions $\Phi$, $\Lambda$, or $\Lambda_Q$ are slowly varying then the conclusions of \cref{cor:occupation_measure} hold with $\beta=2$. Even if the relationship $\Lambda \simeq \Phi^2$ always holds for simple random walk on groups, other values of $\beta$ are certainly relevant for heavy-tailed random walks \cite{saloff2016random,saloff2015random}.

\begin{proof}[Proof of \cref{cor:occupation_measure}]
We may assume that $G$ is amenable, the claim following trivially from Kesten's theorem otherwise, so that $\Gamma$ is unimodular.
Suppose that there exists a slowly varying function $\tilde \Lambda$ with $\Phi^\beta \lesssim \tilde \Lambda \lesssim \Lambda_Q$. As in the proof of \cref{cor:PittetSaloffCoste}, this implies that there exists a constant $C$ such that $\Phi^{-1}(\eps) \leq C \tilde \Lambda^{-1}(\eps^\beta/C)$ for every $0<\eps\leq 1$, and we deduce from \cref{thm:main_inequality_general} that
\[
\P\left(d(X_0,X_k) \leq r\right) 
\leq 2\exp\left[ - \frac{\log 2}{2}\max\left\{\ell \geq 0 :  \frac{\ell \log 2}{\tilde \Lambda(C_2 2^{\ell}\tilde \Lambda^{-1}(c_1/r^\beta))} \leq k \right\} \right].
\]
Since $\tilde \Lambda$ is slowly varying, for each $\eps>0$ there exists a constant $r_0$ such that if $r\geq r_0$ then
\[
\tilde \Lambda\Bigl(C_2 2^{\ell}\tilde \Lambda^{-1}(c_1/r^\beta)\Bigr) \geq 2^{-\eps \ell} \tilde \Lambda\Bigl(\tilde \Lambda^{-1}(c_1/r^\beta)\Bigr) = 2^{-\eps\ell}\cdot\frac{c_1}{r^\beta}.
\]
Absorbing the linear term $\ell$ into the exponential $2^{-\eps \ell}$ at the cost of an $\eps$-dependent constant, it follows that for each $\eps>0$ there exists a constant $C_2(\eps)$ such that
\[
\P\left(d(X_0,X_k) \leq r\right) 
\leq 2\exp\left[ - \frac{\log 2}{2}\max\left\{\ell \geq 0 : C_2(\eps)2^{2\eps \ell} r^\beta \leq k\right\} \right].
\]
Since $\eps>0$ was arbitrary, this is easily seen to imply that for each exponent $p\geq 1$ there exists a constant $C_p$ such that
\[
\P\left(d(X_0,X_k) \leq r\right) \leq C_p\left(\frac{r}{k^{1/\beta}}\right)^p
\]
as claimed. The second claimed inequality follows from this one by an elementary calculation.
\end{proof}

The regularity hypothesis used in this corollary holds in particular if either $\Lambda$ or $\Phi$ is roughly slowly varying. We conjecture that this hypothesis always holds for groups of superpolynomial growth:

\begin{conjecture}
\label{conj:slowly_varying}
Let $G$ be a connected, locally finite, transitive graph. If $G$ has superpolynomial growth then its spectral and isoperimetric profiles $\Phi$ and $\Lambda$ are both roughly slowly varying.
\end{conjecture}

We would actually make the stronger conjecture that $\Phi$ and $\Lambda$ are slowly varying rather than merely roughly slowly varying, but the distinction is not important for our applications.

\medskip

Let us now give some brief justification for this conjecture. First, it is true in all examples where $\Phi$ and $\Lambda$ have been computed explicitly: see \cite[Table 1]{bendikov2012spectral} for an overview. Second, it is a consequence of the landmark work of Breuillard, Green, and Tao \cite{breuillard2011structure} and its generalizations to transitive graphs by Tessera and Tointon \cite{MR4253426} that if $G$ is a connected, transitive, locally finite graph of superpolynomial volume growth then $\operatorname{Gr}(2n)/\operatorname{Gr}(n)\to \infty$ as $n\to\infty$ and hence that the inverse growth function $\operatorname{Gr}^{-1}$ is slowly varying. Thus, $\Phi$ is roughly slowly varying whenever the Coulhon--Saloff-Coste inequality \eqref{eq:Coulhon_Saloff_Coste} is sharp in the sense that $\Phi \simeq 1/\operatorname{Gr}^{-1}$.  Let us also mention that all the groups constructed in \cite{brieussel2015speed} have isoperimetric and spectral profiles that are roughly slowly varying, being of the form $f(\log x)/\log x$ and $(f(\log x)/\log x)^2$ for $f$ an increasing function with $f(x)/x$ monotone.  In \cref{sec:multilateral} we prove that the conjecture also holds for a number of classical examples where the profiles $\Phi$ and $\Lambda$ are \emph{not} known explicitly, including the first Grigorchuk group and Thompson's group $F$.


\medskip

\noindent
\textbf{Sharp small-ball estimates for groups of stretched exponential growth.}
If one assumes not only that $\Lambda$ is slowly varying but that $\Lambda(2^n)$ is \emph{doubling}, meaning that there exists a constant $c>0$ such that $\Lambda(2^{2n})\geq c\Lambda(2^n)$ for every $n\geq 1$, then one can extract the following very strong small-ball estimate from \cref{thm:main_inequality}.  This assumption holds if e.g.\ $\Lambda(n)\simeq (\log n)^{-\alpha}$ for some $\alpha>0$, as is the case in many examples of groups of superpolynomial growth.

\begin{corollary}
\label{cor:small_ball}
Let $G=(V,E)$ be a connected, locally finite, transitive graph, let $\Gamma \subseteq \Aut(G)$ be a closed transitive, unimodular subgroup of $\Aut(G)$, and suppose that $Q:V\times V\to[0,1]$ is a stochastic matrix that is symmetric and $\Gamma$-diagonally invariant in the sense that $Q(\gamma x, \gamma y)=Q(x,y)$ for every $x,y\in V$ and $\gamma \in \Gamma$. Let $X=(X_k)_{k\geq 0}$ be the random walk on $V$ with transition matrix $Q$. 
Let $\Psi(n)= n / \Lambda_Q(2^n)$. If $\Lambda_Q(2^n)$ is doubling and $\beta >0$ is such that $\Lambda_Q \gtrsim \Phi^\beta$ then there exists a positive constant $c$ such that
\[
\P(d(X_0,X_k)\leq r) \leq \exp\left[ - c \min\left\{\frac{k}{r^\beta}, \Psi^{-1}(ck)\right\} \right]
\]
for every $k,r\geq 1$. 
\end{corollary}

Again, if 
 $Q(x,y)>0$ whenever $x$ and $y$ are adjacent in $G$ then $\Lambda_Q \gtrsim \Lambda \gtrsim \Phi^2$ by Cheeger's inequality, so we can always apply this corollary with $\beta=2$ under the assumption that $\Lambda_Q(2^n)$ is doubling.
 As discussed around \eqref{eq:Grigoryan} above, the assumption that $\Lambda(2^n)$ is doubling ensures by the results of \cite{MR1418518,coulhon1997diagonal} that
\[
P^{2n}(o,o) \simeq \exp[-\Psi^{-1}(n)],
\]
so that the second term in the minimum can be thought of as accounting for the trivial fact that $\P(d(X_0,X_k)\leq r) \geq \P(X_k=X_0)$.

\begin{proof}[Proof of \cref{cor:small_ball} given \cref{thm:main_inequality}]
The assumption that $\Lambda(2^n)$ is doubling ensures that there exist constants $c_1,c_2>0$ such that
\[\Lambda_Q(2^{\ell+1}\Phi^{-1}(c/r)) \geq c_1 \min\{\Lambda(2^{\ell}),\Lambda_Q(\Phi^{-1}(c/r))\} \geq c_2 \min\{\Lambda_Q(2^\ell),r^{-\beta}\}, \]
where the second inequality follows from a second application of the doubling property together with the assumption that $\Lambda_Q \gtrsim \Phi^\beta$. Thus, \cref{thm:main_inequality} implies that there exists a constant $C<\infty$ such that
\[
\P(d(X_0,X_k)\leq r) \leq 2\exp\left[ - \frac{\log 2}{2}\max\left\{\ell \geq 0 : k\geq \frac{C \ell}{\Lambda(2^{\ell})} \text{ and } k \geq C\ell r^\beta \right\} \right].
\]
This is easily seen to be equivalent to the claim.
\end{proof}

As before, the regularity assumption we make on $\Lambda$ is known to hold in every superpolynomial growth example where $\Lambda$ has been computed explicitly \cite[Table 1]{bendikov2012spectral} and is plausibly true for \emph{every} transitive graph of superpolynomial growth.  This is related to Grigorchuk's \emph{gap conjecture} \cite{MR3174281}, which states that $\operatorname{Gr}(r)\gtrsim \exp[\sqrt{r}]$ for every finitely generated group of superpolynomial growth. Thus, together with the Coulhon--Saloff-Coste inequality \eqref{eq:Coulhon_Saloff_Coste} and Cheeger's inequality \eqref{eq:Cheeger}, the extension of the gap conjecture to transitive graphs would imply in particular that $\Lambda(2^n)\gtrsim n^{-4}$ whenever $G$ has superpolynomial growth, and from here it would seem very reasonable that $\Lambda(2^n)$ is always doubling when $G$ has superpolynomial growth. On the other hand, the best quantitative results on the gap conjecture are very far from verifying this \cite{shalom2010finitary}. 
In  \cref{sec:multilateral} we prove that this regularity property is satisfied for a class of examples including the first Grigorchuk group.

\begin{remark}
 The bound of \cref{cor:small_ball} is optimal for the simple random walk on the lamplighter group $\Z_2 \wr \Z$, which has $\Lambda(n) \simeq (\log n)^{-2}$. Indeed, the probability that the simple random walk on $\Z$ stays in $[-r,r]$ for $k$ steps is of order $\exp[-\Theta(k/r^2)]$, and on this event the lamplighter walk has distance at most $O(r)$ from the origin at time $k$. The upper bound in this example can also be deduced as an application of \cref{cor:small_ball} since $\Lambda \simeq (\log n)^{-2}$ satisfies the required regularity properties.
\end{remark}

\section{Proof of the main theorem}

In this section we prove \cref{thm:main_inequality,thm:main_inequality_general}.
  In contrast to previous works on diffusive lower bounds for random walks on groups, our proof does not rely on embeddings into Hilbert space or martingales in any way. Our proof is instead inspired by the theory of \emph{actions on spaces with measured walls} \cite{cherix2004spaces,mukherjee2023haagerup} developed for the study of the Haagerup property, and can be interpreted in terms of a certain equivariant embedding of the graph into $L^1$. (More accurately, we construct a different embedding for each scale.) Although these perspectives inspired our work, we will formulate the proof in an elementary way without reference to measured walls or $L^1$ embeddings since these notions are not really needed in the proof.


Let $Q$ be a symmetric stochastic matrix defined on a countable set $V$.
Following \cite{hermon2023relaxation}, we introduce for each $\ell\geq 1$ and $n\geq 0$ the quantity $\chi_Q(n,\ell)$ defined by
\[
\chi_Q(n,\ell) = \inf \left\{k\geq 0: \|Q^k\mathbbm{1}_W\|_2^2 \leq 2^{-\ell} \|\mathbbm{1}_W\|_2^2 \text{ for every $W \subseteq V$ with $|W|\leq n$} \right\}.
\]
This quantity can be related to the probability that a random walk started at a uniform point of a set $W$ still belongs to  $W$ at time $k$ using Cauchy-Schwarz: Letting $\P_W$ denote the law of the random walk started at a uniform random point of $W$, we have that 
\[
\P_W(X_k \in W) = \frac{1}{|W|}\langle \mathbbm{1}_W, Q^n \mathbbm{1}_W \rangle \leq \frac{\|Q^n \mathbbm{1}_W\|_2\|\mathbbm{1}_W\|_2}{\|\mathbbm{1}_W\|_2^2} \leq 2^{-\ell/2} 
\]
for every $k,\ell\geq 1$ such that $k\geq \chi_Q(|W|,\ell)$. Our proof relies heavily on the following proposition, which was proven by Hermon in \cite[Proposition 1.6]{hermon2023relaxation}. This proposition is closely related to the inequalities proven in \cite[Section 3]{HermonHutchcroftIntermediate} and can be thought of an as infinite-volume analogue of the $L^\infty$ mixing time bounds of Goel, Montenegro, and Tetali \cite{MR2199053}.

\begin{prop}
\label{prop:chi_n_ell}
 Let $Q$ be a symmetric stochastic matrix defined on a countable set $V$. Then 
\[
\chi_Q(n,\ell) \leq \frac{\ell \log 2}{\Lambda_{Q}(n2^{\ell+1})}
\]
for every $n,\ell \geq 1$.
\end{prop}

The assumption that $Q$ is symmetric is not really needed, but lets us continue assuming that the counting measure is stationary. If one formulates the definitions correctly then it suffices that $Q$ is reversible. See \cite{hermon2023relaxation} for these more general statements. In \cite{hermon2023relaxation} it is also proven that $\chi_Q(n,1)\geq c/\Lambda(n)$ for some positive universal constant $c$, so that the inequality is always sharp for small values of $\ell$.

\begin{proof}[Proof of \cref{thm:main_inequality_general}]
Let $\operatorname{Haar}$ denote the Haar measure on $\Gamma$, normalized so that $\operatorname{Haar}(\operatorname{Stab}_o)=1$. The assumption that $\Gamma$ is unimodular ensures that
\begin{equation}
\label{eq:Haar1}
\operatorname{Haar}(\{\gamma \in \Gamma : \gamma x = y \}) = 1 \qquad \text{ for all $x,y\in V$}
\end{equation}
(indeed, this can be taken as the \emph{definition} of unimodularity) and moreover that
\begin{equation}
\label{eq:Haar2}
\operatorname{Haar}(\{\gamma \in \Gamma : \gamma e_1 = e_2 \}) = \frac{\mathbbm{1}(e_2 \in \Gamma e_1)}{\deg_{e_1}(o)},
\end{equation}
where we define $\deg_{e}(o):=\#\{e' \in E^\rightarrow_o: e' \in \Gamma e\}$.

\medskip

Let $n\geq 1$, let $W_n$ be a finite set of vertices with $|W_n|\leq n$ such that $|\partial_E W_n|=\deg(o)\Phi(n) |W_n|$, and write $\partial^\rightarrow_E W_n$ for the set of oriented edges $e$ with $e^-\in W_n$ and $e^+\notin W_n$. We define a metric $d_n$ on $V$ by
\begin{align*}
d_n(x,y) &= \frac{1}{2}\operatorname{Haar}\{\gamma \in \Gamma : x\in \gamma W_n \text{ and } y \notin \gamma W_n \text{ or } x\notin \gamma W_n \text{ and } y \in \gamma W_n \}\\
&=\frac{1}{2}\operatorname{Haar}\{\gamma \in \Gamma : x\in \gamma W_n\} + \frac{1}{2}\operatorname{Haar}\{\gamma \in \Gamma : y \in \gamma W_n\} - \operatorname{Haar}\{\gamma \in \Gamma : x,y\in \gamma W_n\} \\
&= |W_n| - \operatorname{Haar}\{\gamma \in \Gamma : x,y \in \gamma W_n \}.
\end{align*}
The fact that this satisfies the triangle inequality follows from the fact that is a non-negative linear combination of the pseudometrics $\mathbbm{1}(x\in \gamma W_n \text{ and } y \notin \gamma W_n \text{ or } x\notin \gamma W_n \text{ and } y \in \gamma W_n)$. We also define the normalized metric $\hat d_n(x,y)$ by
\[
\hat d_n(x,y) = \frac{d_n(x,y)}{\max_{u\sim v} d_n(u,v)},
\]
where the max is taken over all pairs of vertices that are neighbours in $G$. This normalization ensures that $\hat d_n$ is bounded above by the graph distance on $G$. 
To proceed, we will argue that $d_n(X_0,X_k)$ is close to the maximal distance $|W_n|$ with high probability when $k$ is sufficiently large as a function of $n$ (depending on the spectral profile), and that the normalizing factor is bounded by a quantity depending on the isoperimetric profile.

\medskip

\noindent \textbf{Bounding the normalization factor.}
The normalization factor used to define $\hat d_n$ can be bounded in terms of the isoperimetric profile: For each oriented edge $e$ of $G$ we have that
\begin{multline*}
d_n(e^-,e^+)=\frac{1}{2}\operatorname{Haar}(\{\gamma \in \Gamma : e \in \partial_E^\rightarrow (\gamma W_n)\})+\frac{1}{2}\operatorname{Haar}(\{\gamma \in \Gamma : e^\leftarrow \in \partial_E^\rightarrow (\gamma W_n)\}) \\
= \frac{1}{2}\sum_{e' \in \partial_E^\rightarrow W_n} \left[\operatorname{Haar}(\{\gamma \in \Gamma : e = \gamma e'\})+ \operatorname{Haar}(\{\gamma \in \Gamma : e^\leftarrow = \gamma e'\})\right] 
\\
= \frac{|\partial_E^\rightarrow W_n \cap (\Gamma e)|}{2\deg_e(o)}+\frac{|\partial_E^\rightarrow W_n \cap (\Gamma e^\leftarrow)|}{2\deg_{e^\leftarrow}(o)} \leq \frac{1}{2}\left[\frac{\deg(o)}{\deg_e(o)}+\frac{\deg(o)}{\deg_{e^\leftarrow}(o)}\right]\Phi(n)|W_n|,
\end{multline*}
where $e^\leftarrow$ denotes the reversal of $e$. (In fact unimodularity implies that $\deg_e(o)=\deg_{e^\leftarrow}(o)$, but we will not need this.)
Since the edge $e$ was arbitrary, it follows that
\begin{equation}
\label{eq:normalizing_constant}
\max_{u\sim v} d_n(u,v) \leq \left(\max_{e\in E_o^\rightarrow} \frac{\deg(o)}{\deg_e(o)} \right) \Phi(n)|W_n|.
\end{equation}

\medskip
\noindent \textbf{The walk is at near-maximal distance with high probability.} This will be proven via a first moment calculation.
We have by Fubini's Theorem that
\[
\E [|W_n|-d_n(X_0,X_k)] = \E\left[\operatorname{Haar}\{\gamma \in \Gamma : X_0,X_k\in \gamma W_n\}\right] = \int_\gamma \sum_{x\in V} \mathbbm{1}(o,x\in \gamma W_n)Q^k(o,x) \dif \operatorname{Haar}(\gamma),
\]
and using the assumption that $Q$ is $\Gamma$-diagonally invariant we obtain that
\begin{align*}
\int_\gamma \sum_{x\in V} \mathbbm{1}(o,x\in \gamma W_n)Q^k(o,x) \dif \operatorname{Haar}(\gamma) &= \int_\gamma \sum_{x\in V} \mathbbm{1}(\gamma^{-1}o,\gamma^{-1}x\in  W_n)Q^k(o,x) \dif \operatorname{Haar}(\gamma) \\
&=\int_\gamma \sum_{x\in V} \mathbbm{1}(\gamma^{-1}o,\gamma^{-1}x\in  W_n)Q^k(\gamma^{-1}o,\gamma^{-1}x) \dif \operatorname{Haar}(\gamma)\\
&=\int_\gamma \sum_{x\in V} \mathbbm{1}(\gamma^{-1}o, x\in  W_n)Q^k(\gamma^{-1}o,x) \dif \operatorname{Haar}(\gamma)\\
&=  \sum_{w \in W_n} \P_w(X_k \in W_n) = |W_n|\P_{W_n}(X_k\in W_n),
\end{align*}
where the penultimate line follows from the fact that $\gamma:V\to V$ is a bijection and where we used \eqref{eq:Haar1} in the last line.
As such, it follows from the definition of $\chi_Q(n,\ell)$ that if $k\geq \chi_Q(n,\ell)$ then
\[\E [|W_n|-d_n(X_0,X_k)] \leq 2^{-\ell/2}|W_n|.\]
In particular, we have by Markov's inequality that 
\begin{equation}
\label{eq:far_away_high_prob}
\P(d_n(X_0,X_k) \leq |W_n|/2) \leq 2^{1-\ell/2}
\end{equation}
whenever $k\geq \chi_Q(n,\ell)$.

\medskip
\noindent \textbf{Completing the proof.} It follows from \eqref{eq:normalizing_constant}, \eqref{eq:far_away_high_prob} and the fact that $d(X_0,X_k) \geq \hat d_n(X_0,X_k)$ that
\begin{equation}
\label{eq:nearly_done}
\P\left(d(X_0,X_k)\leq \frac{1}{2}\left(\max_{e\in E_o^\rightarrow} \frac{\deg(o)}{\deg_e(o)} \right)^{-1} \frac{1}{\Phi(n)}\right) \leq 2^{1-\ell/2}
\end{equation}
whenever $k\geq \chi_Q(n,\ell)$. Applying \cref{prop:chi_n_ell} to bound $\chi_Q(n,\ell) \leq \ell \log 2/ \Lambda_Q(n2^{\ell+1})$ yields that \eqref{eq:nearly_done} holds whenever $k\geq\ell \log 2/ \Lambda_Q(n2^{\ell+1})$, which is equivalent to the claimed inequality.
\end{proof}

\section{Regularity assumptions are satisfied by multilateral groups}
\label{sec:multilateral}

 We now describe a class of groups for which the regularity assumptions needed to apply \cref{cor:occupation_measure,cor:small_ball} hold. Recall that two groups $\Gamma_1,\Gamma_2$ are said to be \textbf{commensurable}, denoted $\Gamma_1 \approx \Gamma_2$, if they have isomorphic subgroups of finite index. A group $G$ is said to be \textbf{multilateral} if it is commensurable with a non-trivial direct power of itself, i.e.\ if $\Gamma \approx \Gamma^m$ for some $m>1$.  More generally, we define a group $\Gamma$ to be \textbf{supermultilateral} if it has a \emph{subgroup} $H$ such that $H \approx \Gamma^m$ for some $m\geq 1$. Some of the best known examples of groups of superpolynomial growth are supermultilateral, including the first Grigorchuk group \cite[Proposition 6.1]{grigorchuk2006groups} and Thompson's group $F$ \cite[Section 4.2]{bader2014weak}; the asymptotics of $\Phi$ and $\Lambda$ are not known explicitly in either example. (Indeed, it a major open problem whether or not Thompson's group $F$ is amenable. It is known that if it is amenable then its isoperimetric and spectral profiles decay extremely slowly \cite{moore2013fast}.)

\begin{prop}
\label{prop:multilateral}
Let $\Gamma$ be a finitely generated group, let $G$ be a Cayley graph of $\Gamma$, and let $\Phi$, $\Lambda$, and $\operatorname{Gr}^{-1}$ denote the isoperimetric profile, spectral profile, and inverse growth function of $G$. If $\Gamma$ is supermultilateral then the functions $\Phi(2^n)$, $\Lambda(2^n)$, and $\operatorname{Gr}^{-1}(2^n)$ are all doubling.
\end{prop}

\begin{proof}[Proof of \cref{prop:multilateral}]
It is a standard fact that if $G_1$ and $G_2$ are Cayley graphs of commensurable groups then $G_1$ and $G_2$ are quasi-isometric and satisfy $\Phi_{G_1} \simeq \Phi_{G_2}$, $\Lambda_{G_1} \simeq \Lambda_{G_2}$, and $\operatorname{Gr}^{-1}_{G_1} \simeq \operatorname{Gr}^{-1}_{G_2}$. Moreover, if $G$ is a Cayley graph of a finitely generated group $\Gamma$ and $G'$ is a Cayley graph of a finitely generated subgroup $H$ of $\Gamma$ then $\Phi_G \gtrsim \Phi_H$, $\Lambda_G \gtrsim \Lambda_H$, and $\operatorname{Gr}^{-1}_G \lesssim \operatorname{Gr}^{-1}_H$. Finally, it follows from a theorem of Coulhon, Grigor'yan, and Levin \cite{coulhon2003isoperimetric} that
\[
\Phi_{G^m}(n) \simeq \Phi_G(n^{1/m}), \qquad \Lambda_{G^m}(n) \simeq \Lambda_{G}(n^{1/m}), \qquad \text{ and } \qquad \operatorname{Gr}^{-1}_{G^m}(n) \simeq \operatorname{Gr}^{-1}_{G}(n^{1/m}).
\]
Thus, if $\Gamma$ is supermultilateral then
\[
\Phi(n) \gtrsim \Phi(n^{1/m}), \qquad \Lambda(n) \gtrsim \Lambda(n^{1/m}), \qquad \text{ and } \qquad \operatorname{Gr}^{-1}(n) \lesssim \operatorname{Gr}^{-1}(n^{1/m})
\]
for some integer $m\geq 1$. As such, it suffices to prove that if $f$ is an increasing function with $f(n) \lesssim f(n^{c})$ for some $0\leq c<1$ then $f(2^n)$ is doubling. To see this, note that we can iterate the inequality $f \lesssim f(n^c)$ some finite number of times to obtain that $f(n) \lesssim f(n^{1/4})$, so that there exist positive constants $C_1$ and $C_2$ such $f(n)\leq C_1 f(C_2 n^{1/4})$ for all $n\geq 1$. If $n$ is sufficiently large then $2^{n}\geq C_2 (2^{2n})^{1/4}$ and it follows that if $n$ is sufficiently large then
$f(2^{2n}) \leq C_1 f(2^n)$ as claimed.
\end{proof}

The following lemma shows that the regularity condition guaranteed by \cref{prop:multilateral} is also sufficient to apply \cref{cor:occupation_measure}.

\begin{lemma}
\label{lem:log_doubling}
Let $f:\N\to (0,\infty)$ be monotone. If $f(2^n)$ is doubling  then $f$ is roughly slowly varying. 
\end{lemma}

\begin{proof}[Proof of \cref{lem:log_doubling}]
The function $\tilde f$ defined by
\[\tilde f(n) := f(2^{\lfloor \log_2 n\rfloor})^{\log_2 n - \lfloor \log_2 n\rfloor}f(2^{\lfloor \log_2 n\rfloor+1})^{\lfloor \log_2 n\rfloor+1-\log_2 n} \]
is easily verified to be slowly varying and satisfy $f \simeq \tilde f$ whenever $f$ is a positive monotone function and $f(2^n)$ is doubling.
\end{proof}

\section*{Acknowledgements}
The work was supported by NSF grant DMS-1928930. We thank Jonathan Hermon, Igor Pak, Christophe Pittet, Laurent Saloff-Coste, Matt Tointon, and Tianyi Zheng for helpful discussions.

 \setstretch{1}
 \footnotesize{
  \bibliographystyle{abbrv}
  \bibliography{unimodularthesis.bib}
  }

\end{document}